\documentclass[a4paper,11pt]{amsart}
\usepackage{amsmath,amssymb,colordvi}

\newtheorem{theorem}{Theorem}
\newtheorem{corollary}[theorem]{Corollary}
\newtheorem{lemma}[theorem]{Lemma}
\newtheorem{example}[theorem]{\it Example}
\newtheorem{proposition}[theorem]{Proposition}
\newtheorem{definition}[theorem]{Definition}

\newtheorem{remark}[theorem]{\it Remark}

\font\tenBb=msbm10 \font\sevenBb=msbm7 \font\fiveBb=msbm5
\newfam\Bbfam \textfont\Bbfam=\tenBb \scriptfont\Bbfam=\sevenBb
\scriptscriptfont\Bbfam=\fiveBb

\def\Bb{\fam\Bbfam\tenBb}
\def\R{{\Bb R}}
\def\C{{\Bb C}}
\def\N{{\Bb N}}

\begin{document}
\title
[ Sharp extensions for  solutions of wave equations]{Sharp extensions for convoluted solutions of wave equations}

\author[P. J. Miana]{Pedro J. Miana}
\address{Departamento de Matem\'{a}ticas e I. U. M. A.
Universidad de Zaragoza, 50009 Zaragoza, Spain}
\email{pjmiana@unizar.es}

\thanks{}

\author[V. Poblete]{Ver\'onica Poblete}
\address{Universidad de Chile, Departamento de
Matem\'atica, Facultad de Ciencias, Las Palmeras 3425 \~{N}u\~{n}oa,
Santiago-Chile.}
 \email{vpoblete@uchile.cl}
\thanks{ The first author has been partly supported by Project MTM2010-16679, DGI-FEDER, of the MCYTS, Spain, Project E-64, D.G. Arag\'on, and JIUZ-2012-CIE-12, Universidad de Zaragoza, Spain. The second author is partially financed by  Proyecto Fondecyt 1110090}

 \begin{abstract}
In this paper we give sharp extensions of convoluted solutions of wave equations in abstract Banach spaces. The main technique is to use the algebraic structure, for convolution products $\ast$ and $\ast_c$,  of these solutions which are defined by a version of the Duhamel's formula. We define  algebra homomorphisms, for the convolution product $\ast_c$, from a certain set of test-functions  and apply our results to concrete examples of   abstract wave equations.
\end{abstract}

\date{}

\maketitle

\section{Introduction}
\setcounter{theorem}{0}
\setcounter{equation}{0}

We consider the following evolution problem
\begin{equation}\label{equa}
\begin{cases}
\displaystyle{\partial u\over \partial t}=\Delta u, &(t,x)\in \R^+ \times V\\
u(0,x)= u_0(x), & x\in V;\\
u(t, \cdot)\vert_{\partial V}= g(t, \cdot), & t>0.\\
\end{cases}
\end{equation}
\noindent with $V$ is an open set  in $\R^n$.
In \cite{[Du]},  Duhamel proposed the following formula to express the solution of problem (\ref{equa}).
$$
u(t,x)=\int_0^t{\partial \over \partial t}u(\lambda, t-\lambda,x)d\lambda, \quad t>0,
$$
where $u(\lambda,t,x)$ is a  solution of problem (\ref{equa}) for a particular function $g(\cdot, \lambda_0)$ fixed $\lambda_0$. Note that this formula reduces the Cauchy problem for an in-homogeneous partial differential equation to the Cauchy problem for the
corresponding homogeneous equation. This formula (known as Duhamel's principle) is used in partial differential equations and have been studied in  a large number of papers. We only mention here the paper \cite{DL} where an abstract point of way is treated to give   solutions of in-homogeneous differential equation in a Banach space $X$. Recently a fractional Duhamel's principle have been introduced to study  the Cauchy
problem for  inhomogeneous fractional  differential equations in \cite{Sa}.

Let $k:(0,T) \to \C$  a local integrable function, $X$ be a Banach space and $x\in X.$ The Cauchy problem
$$
\begin{cases}
\displaystyle{v'(t) = Av(t) + \int_0^tk(s)xds,} \quad &0<t<T,\\
 v(0) = 0,
\end{cases}
$$
is usual called $K$-convoluted Cauchy problem where $K(t)=\int_0^tk(s)x$ for $0<t<T$.
If there exists a solution of the abstract
Cauchy problem $u'(t) = Au(t)$ for $ 0<t < T,$ $u(0) = x$
then, as usual for a nonhomogeneous equation, then we have $v=u\ast K$ with $\ast$ is the usual convolution in $\R^+$.
The theory of local $k$-convoluted semigroup was introduced in \cite{C, CL} and extends the classical theory of $C_0$-semigroups, see a complete treatment in \cite[Section 1.3.1]{[MF]}, \cite[Chapter 2]{ko} and  other details \cite{KLM, K1, [KP]} and reference in. One of the interesting consequence of this theory is the extension property of the solution of  $K_1$-convoluted problem, using a version of Duhamel's formula, which increases  the regularity of the problem, see \cite{CL,ko}.

Now we consider the $K$-convoluted second order Cauchy problem (or $K$-convoluted wave equation),
\begin{equation}\label{(ACP_2)}
\begin{cases}
\displaystyle{v''(t) = Av(t) + \int_0^tk(s)xds+ \int_0^t(t-s)k(s)yds,} \quad &0<t<T,\\
 v(0) = 0, \quad v'(0)=0,
\end{cases}
\end{equation}
\noindent  for $x,y\in X.$ The existence  of a unique solution of problem (\ref{(ACP_2)}) is closely connected with the existence of a unique $k$-convoluted mild solution of the wave equation  $u''(t) = Au(t)$ for $ 0<t < T,$ $u(0) = x$ and $u'(0)=y$ (see for example \cite[Theorem 2.1.3.9, Corollary 2.1.3.16]{ko}).

The main objective of this paper is to illustrate  the algebraic structure of this family of operators (above mentioned as local $k$-convoluted mild solution) and also known as local $k$-convoluted cosine functions.  One condition in the definition of local $k$-convoluted cosine functions followed in this paper (Definition \ref{cos-con} (ii)) may be interpreted as a Durhamel formula for the wave problem. Other equivalent definitions of local $K$-convoluted cosine functions, using the composition property of local $k$-convoluted cosine functions (\cite[Theorem 2.4]{K2}) or the Laplace transform (\cite[Definition 2.1]{K1})  show   this algebraic aspect in a straightforward way.

The usual convolution $\ast$ on $\R^+$, see (\ref{convo}), and the cosine convolution $\ast_c$ see (\ref{cosine}), are the starting technical tools to show some of our results. In Section 2, we prove some identities (in particular Lemma \ref{lemaclave}) for these convolution products which we apply in later sections.

The  extension formula in Theorem \ref{exten}, which improves some previous results, \cite[Theorem 3.1]{K2} and \cite[Theorem 2.1.1.30 (v)]{ko}, shows clearly how the  algebraic properties fits perfectly with the inner structure of $k$-convoluted cosine function. This theorem and the composition property of $k$-convoluted cosine function allow to define algebra homomorphism  from a particular set of test-functions in Theorem \ref{main2}. In the global case these algebra homomorphisms were considered in  \cite[Theorem 6.5]{KLM}.

The special case of $k(t)={t^{\alpha-1}\over \Gamma(\alpha)}$ with $\alpha >0$ defines the $\alpha$-times integrated cosine function. Originally they were the first example of $k$-convoluted cosine function. Note that in this case the algebra homomorphism (for $\ast_c$) are known as  distribution cosine function  and were introduced  in \cite{KTW} and studied in \cite{KMTW, M0}. Algebra homomorphisms from cosine convolution algebras and $\alpha$-times integrated cosine functions were treated deeply in \cite{M2}.

From other point of view, some particular connections between $k$-convoluted cosine function and vector valued homomorphisms (in particular ultradistributions and hyperfunction sines) have been studied in \cite[Section 5]{[KP2]} and \cite[Section 3.6.3]{ko}. As these authors comment, it is not clear how to obtain the corresponding results (similar to distribution cosine functions) in the case of ultradistribution and hyperfunction sines, see \cite[Remark 14]{[KP2]}. Our results in Section 4 show a path to define algebra homomorphisms using $k$-convoluted cosine functions.

Similar results as Theorem \ref{exten} and Theorem \ref{main2} may be proved for local $k$-convoluted $C$-cosine functions where $C$ is an injective operator (see \cite[Definition 1.1]{K2}); for $C=I_X$, this definition corresponds to local $k$-convoluted $C$-cosine families, see Remark \ref{c-op}.

In the last section, we apply our results to several particular examples of global and local $k$-convoluted cosine functions which have appeared in the context of this literature.

Similar results for  $k$-convoluted semigroups, distribution semigroups and algebra homomorphisms (for $\ast$)  hold and may be found in \cite{KMS}. Remind that the abstract Weierstrass formula gives (global) $\tilde k$-convoluted semigroups subordinated to (global) $k$-convoluted cosine functions, where
$$
\tilde k(t)=\int_0^\infty {s e^{-s^2\over 4t}\over 2\sqrt{\pi}t^{3/2}}k(s)ds, \qquad t\ge 0,
$$
see \cite[Theorem 11]{[KP2]}.

\section{Identities for convolutions products and k-Test function spaces}
\setcounter{theorem}{0}
\setcounter{equation}{0}

 Let $L^1_{loc}(\mathbb{R}^+)$ the set of locally integrable functions on $\mathbb{R}^+.$ For $f,g \in L^1_{loc}(\mathbb{R}^+),$ we consider the usual convolution given by
 \begin{equation}\label{convo}
 (f*g)(t) = \int_0^t f(t-s) g(s) ds\,, \ \ t \ge 0.
 \end{equation}
For $t \ge 0,$  we denote by $\chi$ the constant function equals to $1$, i.e., $\chi(t):=1,$ by  $I(t):=(\chi\ast\chi)(t)=t$. Consequently,
 $$
 (\chi*f)(t) = \int_0^t f(s) ds\,, \ \  (I*f)(t) = \int_0^t (t-s) f(s) ds\,, \ \ \ t \ge 0.
 $$

  We write $f^{*1} = f,\,$ $\,f^{*2} = f*f$ and   $f^{*n} = f*f^{*(n-1)}$ for $n \in \mathbb{N}$. We also follow the notation $\circ$ to denote the dual convolution product of $\ast$ given by
$$\displaystyle (f\circ g)(t)=\int_t^\infty f(s-t)g(s)ds, \ \ \  t\ge0,$$
where $f, \, g\in L^1(\R^+)$. Note that,
$$
\int_0^\infty f(t)\left(g\ast h\right)(t)dt=\int_0^\infty h(t)\left(g\circ f\right)(t)dt, \qquad f, g, h\in L^1(\R^+).
$$
There are interesting equalities between both convolution products, for example
$$
k \circ (f \circ g) = (k*f) \circ g = (f*k) \circ g = f \circ (k \circ g) , \qquad f, g , k\in L^1(\mathbb{R}^+).$$
 The cosine convolution product $(f*_cg)$ is defined by
\begin{equation} \label{cosine}f\ast_c g:= {1 \over 2}(f\ast g
+f\circ g+g\circ f), \qquad f, g\in L^1(\mathbb{R}^+).\end{equation}

For the show of the previous statement and more properties  of this convolution see \cite{KLM, M1, M2}. We denote by $\hat f$ the usual Laplace transform of a function  $f\in L^1_{loc}(\mathbb{R}^+)$ given by
$$
\hat f(\lambda)= \int_0^\infty e^{-\lambda t}f(t)dt,
$$
for $\lambda\in \C$ such that this integral converges.

We denote by  $\mathcal{D}$  the set of $\mathcal{C}^{(\infty)}$ functions with compact support on $\R$.We write by $\mathcal{D}_+$ the set of functions defined by  $\phi_+:[0, \infty) \to \C$, given by $\phi_+(t):=\phi(t)$ for $t\ge 0$ and $\phi\in \mathcal{D}$. Note that if  $\psi $ is a $\mathcal{C}^{(\infty)}$ functions on $[0,\infty)$ and  compact support  then $\psi \in \mathcal{D}_+$, see \cite{[Se]}. All these spaces are topological vector spaces equipped with the topology of uniform convergence on bounded subsets.

\begin{proposition} \label{propi} Let $\phi, \psi \in \mathcal{D}_+$ and $t \ge 0$. Then
\begin{itemize}
\item[(i)]$(\phi'\ast \psi)(t)= (\phi\ast \psi')(t) + \psi(0)\phi(t)-\phi(0)\psi(t)$.
\item[(ii)] $(\phi'\circ \psi)(t)=-\phi(0)\psi(t)- (\phi\circ\psi')(t)$.
\item[(iii)] $(\phi\circ \psi')(t)=-\phi(0)\psi(t)- (\phi'\circ\psi)(t)$.
\item[(iv)] $(\phi'\ast_c \psi)(t)={1\over 2} \, \left[\phi\ast \psi'-\phi\circ \psi'- \psi'\circ \phi\right](t)-\phi(0)\psi(t)$.
    \item[(v)]$(\phi''\ast_c\psi)(t)= (\phi\ast_c \psi'')(t)+\psi'(0)\phi(t)-\phi'(0)\psi(t)$.
\end{itemize}

\end{proposition}

\begin{proof} The part (i) appears in \cite[Proposition 3.1 (iii)]{[Wa]}. Integrating by parts we obtain
$$
(\phi'\circ \psi)(t)=\int_t^\infty \phi'(s-t)\psi(s)ds=-\phi(0)\psi(t)- (\phi\circ\psi')(t), \qquad t\ge 0,
$$
\noindent hence (ii) and (iii) are shown. The item (iv) is straightforward  from the definition of $\ast_c$ and parts (i), (ii) and (iii). Now to show (v), we apply parts (i)-(iv) to get that
\begin{eqnarray*}
(\phi''\ast_c\psi)(t)&=&{1\over 2} \, \left[\phi'\ast \psi'-\phi'\circ \psi'- \psi'\circ \phi'\right](t) - \phi'(0)\psi(t)\\
&=&{1\over 2} \, \left[\phi\ast \psi''+\phi\circ \psi''+ \psi''\circ \phi'\right](t) + \psi'(0)\phi(t)-\phi'(0)\psi(t).
\end{eqnarray*}
\end{proof}

\begin{lemma}{\label{lemaclave}}
Let $0 \le t \le s$ and $f, g \in L^1_{loc}(\mathbb{R}^+)$. Then

$$
\begin{array}{l}
\mbox{(a)} \  \  \displaystyle (\chi*g)(t) (\chi*f)(s) =  \int_{s}^{t+s} g(t+s-r) (\chi*f)(r) dr -  \int_0^{t} f(t + s - r) (\chi*g)(r) dr \\
\\
\mbox{(b)} \  \  \displaystyle (\chi*g)(t) (\chi*f)(s)  = \displaystyle \int_{s - t}^{s} g(t + r  - s) \, (\chi * f)(r)  \, dr  + \displaystyle \int_0^{t}f(r + s - t) \, (\chi * g)(r)  \, dr. \\
\end{array}
$$

\end{lemma}
\begin{proof} First, we show $(a)$. Integrating by parts and using change of variable we have

\vspace{0.3cm}
$\displaystyle \int_0^{t} f(t + s - r) (\chi*g)(r) dr$
$$
\begin{array}{rl}
&  =   - \displaystyle \int_{t}^{t+s} f(t + s -r) dr \  \int_0^{t} g(u)du + \int_0^{t} g(r)  \int_r^{t+s} f(t + s-u) du \, dr \\
& \\
& =  - \displaystyle \int_0^{s} f(r) dr \  \int_0^{t} g(u)du + \int_{s}^{t+s} g(t + s -r)  \int_{t + s -r}^{t+s} f(t + s-u) du \, dr \\
& \\
& =  - \displaystyle \int_0^{s} f(r) dr \  \int_0^{t} g(u)du + \int_{s}^{t+s} g(t + s - r)  \int_{0}^r f(u) du \, dr \\
& \\
& =  - \displaystyle (\chi*f)(s) \  (\chi*g)(t) + \int_{s}^{t+s} g(t + s - r) (\chi*f)(r) \, dr.
\end{array}
$$

Now, we prove $(b)$, using the Fubini theorem and change the variable we obtain
$$
\begin{array}{l}
\displaystyle \int_0^{t} f(t+s-r) (\chi*g)(r) dr  =  \displaystyle \int_0^{t} g(u) \int_u^{t} f(t + s - r) dr \, du  \\
 \\
=  \displaystyle \int_{s - t}^{s} g(u+ t - s) \int_{u+ t- s}^{t} f(t + s - r) dr \, du = \displaystyle \int_{s - t}^{s} g(u + t - s) \int_{s}^{2s-u} f(r) dr \, du\\
 \\
=  \displaystyle \int_{s - t}^{s} g(u + t-s) \int_{0}^{2s-u} f(r) dr \, du
- \displaystyle \int_{s - t}^{s} g(u+t-s) \int_{0}^{u} f(r) dr \, du\\
\\
\ \ \ - \displaystyle \int_{s - t}^{s} g(u+t-s) \int_{u}^{s} f(r) dr \, du\\
\\
= \displaystyle \int_{s}^{t + s} g(t + s - u) (\chi*f)(u)du -  \int_{s - t}^{s} g(u + t  - s) (\chi* f)(u)\, du \\
\\
\ \ \ - \displaystyle \int_{0}^{t} f(r + s - t) (\chi*g)(r) \, dr.
\end{array}
$$
Hence,
$$
\begin{array}{l}
\displaystyle \int_{s}^{t+s} g(t + s - r) (\chi*f)(r)dr - \displaystyle \int_0^{t} f(t + s -r) (\chi*g)(r) dr =  \\
\\
  \displaystyle \int_{s - t}^{s} g(r + t  - s) (\chi* f)(r)\, dr
+ \displaystyle \int_{0}^{t} f(r + s - t) (\chi*g)(r) \, du,
\end{array}
$$
\noindent from $(a)$ follows $(b).$

\end{proof}

\begin{remark} {\rm If in Lemma \ref{lemaclave} we taking $f = g$ in $(a)$  then
$$
 \displaystyle (\chi*f)(t) (\chi*f)(s)  =  \int_{s}^{t+s} f(t+s-r) (\chi*f)(r) dr -  \int_0^{t} f(t + s - r) (\chi*f)(r) dr\,.
$$

If $s = t$ in $(b)$  then
$$
\displaystyle (\chi*g)(t) (\chi*f)(t)  = \displaystyle \int_{0}^{t} g(r) \, (\chi * f)(r)  \, dr  + \displaystyle \int_0^{t}f(r) \, (\chi * g)(r)  \, dr.
$$
We conclude that
$$
\displaystyle [(\chi*f)(t)]^2  = \int_{t}^{2t} f(2t-r) (\chi*f)(r) dr -  \int_0^{t} f(2t - r) (\chi*f)(r) dr = 2 \displaystyle \int_{0}^{t} f(r) \, (\chi * f)(r) dr.
$$
}
\end{remark}

We write by $\hbox{supp}(h)$ the usual support of a function $h$ defined in $\R$.
  The operator $T'_{k}: \mathcal{D}_+ \rightarrow \mathcal{D}_+$ is given by $f\mapsto T'_{k}(f):=k\circ f$. In the case that $0\in \hbox{supp}(k)$, from  \cite[Theorem 2.5]{KLM},  we have that $T'_{k}$ is an injective, linear and continuous homomorphism such that
$$T'_{k}(f\circ g) = f \circ T'_{k}(g),\qquad f,g\in\mathcal{D}_{+}.$$

According to \cite[Definition 2.7]{KLM}, the space $\mathcal{D}_{k}$ is given by $\mathcal{D}_{k}:=T'_{k}(\mathcal{D}_{+})\subset \mathcal{D}_{+} $ and the right inverse map $T'_{k}$ is  $W_{k}:\mathcal{D}_{k} \rightarrow \mathcal{D}_{+}$ defined by
$$f(t)= T'_{k}(W_{k}(f))(t)=\int_{t}^{\infty}k(s-t)W_{k}f(s)ds,\qquad f\in\mathcal{D}_{k}, \quad t\geq0.$$
 It is clear that the subspace $\mathcal{D}_{k}$ is also a topological vector space.

 The following result are show in \cite[Section 3]{KMS}.

 \begin{proposition} \label{equivalencia} Take $k\in L^{1}_{loc}(\mathbb{R}^{+})$ such that  $0\in \hbox{supp}(k)$.  Then
\medskip
  \begin{itemize}
  \item[(i)]  For $a>0,$   $supp(f)\subset [0,a]$ if and only if $supp(W_kf)\subset [0,a]$ for $f\in\mathcal{D}_{k}$.

  \medskip

  \item[(ii)]   $ \mathcal{D}_{k^{\ast n}}\subset \mathcal{D}_{k^{\ast m}}$  and  $W_{k^{\ast m}}f= k^{n-m}\circ W_{k^{\ast n}}f=  W_{k^{\ast n}}(k^{n-m}\circ f)$  for $f\in \mathcal{D}_{k^{\ast n}}$, and $n\ge m \ge 1$.

  \item[(iii)] The space $\displaystyle \mathcal{D}_{k^{\ast \infty}}:=\bigcap_{n=1}^\infty \mathcal{D}_{k^{\ast n}}$ is a topological vector space, $ W_{k^{\ast n}}\in {\mathcal L}(\mathcal{D}_{k^{\ast \infty}})$  and  $k^{\ast n}\circ W_{k^{\ast n}}f=f $
  for $f\in \mathcal{D}_{k^{\ast \infty}}$ and $n\in \N$.
  \end{itemize}

\end{proposition}

\medskip
\begin{example}\label{function} {\rm (i) Let $\alpha >0$ and $ j_\alpha(t):={t^{\alpha-1}\over \Gamma(\alpha)}$.  The map $W_{j_\alpha}$ is the Weyl fractional derivative of order $\alpha$ (usually written by $W^\alpha$), and ${\mathcal D}_{j_\alpha}={\mathcal D}_+$; in the case $\alpha \in \N$, we have  $W^\alpha = (-1)^\alpha{d^\alpha \over dt^\alpha}$ (see \cite{[GM], SKM}).

\medskip

 (ii) Let $\chi_{(0,1)}$ the characteristic function on the interval $(0,1)$. Then the operator $T'_{\chi_{(0,1)}}$ is given by
 $$
 T'_{\chi_{(0,1)}}(f)(t)=\int_t^{t+1}f(s)ds, \qquad f\in {\mathcal D}_+, \quad t\geq0,$$

  ${\mathcal D}_{\chi_{(0,1)}}={\mathcal D}_+$ and
$$
W_{\chi_{(0,1)}}f(t)=-\sum_{n=0}^{\infty} f'(t+n), \qquad f \in {\mathcal D}_+,\quad t\geq0,
$$
see \cite[Section 2]{KLM}.

\medskip

(iii) Let $0<\delta <1$ and $r>0$. The functions $K_\delta$ given by
$$
K_\delta(t):={1\over 2\pi i}\int_{r-i\infty}^{r+i \infty}e^{\lambda t-\lambda^\delta}d\lambda, \qquad t\ge 0,
$$
verify that $\widehat{K_\delta}(\lambda)=e^{-\lambda^\delta} $ for $\lambda \in \C^+$. It is known that $0\in\hbox{supp}(K_\delta)$ (see \cite[p. 107]{[ABHN]}) and we may define $W_{K_\delta}$ and $\mathcal{D}_{K_\delta}$. This function was considered in \cite[Section 5]{[Ba]}, \cite[Example 6.1]{[KP]} and  \cite[Example 6.1]{[KP2]}).

\medskip

(iv) Let
$$
{\mathcal K}(\lambda):={1\over \lambda^2}\prod_{n=0}^\infty{n^2-\lambda\over n^2+\lambda}, \qquad \Re \lambda >0.
$$
Then there exists  a continuous and exponential bounded function ${\kappa}$ in $[0,\infty)$ with $\widehat \kappa= {\mathcal K}$, $0\in\textrm{supp}(\kappa)$ and  ${\mathcal D}_{\kappa} \varsubsetneq  {\mathcal D}_+$, see \cite[Section 3]{KMS}.
The following example was presented in \cite[Section 5]{[Ba]} and appeared later in other references in connection to convoluted semigroups (see \cite[Example 6.1]{[KP]} and  \cite[Example 6.1]{[KP2]}).
}
\end{example}

\section{Local convoluted cosine functions}
\setcounter{theorem}{0}
\setcounter{equation}{0}

The notion of $k-$convoluted cosine functions appears implicitly introduced in \cite{CaL}, as a generalization of the concept of $n-$times integrated cosine functions given in \cite[Section 6]{AK}. We will consider the following definition of local $k-$convoluted cosine as appear in \cite{K1, [KP2]}. As we commented in the Introduction the condition (ii) in the next definition may be considered as a Duhamel's formula for the abstract wave equation.

\begin{definition}{\label {cos-con}}
Let $A$ a closed operator, $k \in L_{loc}^1([0, \tau))$ and $0 < \tau \le \infty$. A strongly continuous operator family $(C_k(t))_{t \in [0,\tau)}$ is said  a local $k-$convoluted cosine function generated by $A$ if
\vspace{0.2cm}

\begin{itemize}
\item[(i)] $C_k(t) A \subset AC_k(t)\,, \ \ \ t \in [0,\tau)$,
\item[(ii)]  for all $x \in X$ and $t \in [0,\tau)\,:$ \ \  $\displaystyle \int_0^t (t-s) C_k(s)x ds \ \in \ D(A)$ and
$$
A \int_0^t(t-s) C_k(s)x ds = C_k(t)x - (\chi * k)(t)x.
$$

\end{itemize}
\end{definition}

\noindent If $A$ generates  a local $k-$convoluted cosine function $(C_k(t))_{t \in [0,\tau)}$ then $C_k(0) = 0.$  Taking $\tau = \infty$ we have that  $(C_k(t))_{t \in [0,\infty)}$ is an exponentially bounded, $k-$convoluted cosine function with generator $A,$ for view its properties, see \cite{KLM, K1, [KP2]}, among others. For exponentially bounded functions  the $k-$convoluted cosine functions are given  in terms of the Laplace transform (see \cite[Definition 2.1]{K1}).

 Plugging $\,k =  j_\alpha\,$ $(\alpha > 0)$  in above definition, we obtain the well-known classes of $\alpha-$times  integrated cosine function. A characterization  of  exponentially bounded $\alpha-$times integrated cosine function in terms of its Laplace transform was obtained in \cite{CCK}. The relationship between $\alpha-$times integrated
cosine function and the operator Bessel function are studied in \cite{G}.
\vspace{0.2cm}

The following proposition is a direct consequence of \cite[Lemma 4.4]{K1}.

\begin{lemma}{\label{propK}}
Let $A$  be a generator of a local $k-$convoluted cosine function $(C_k(t))_{t \in [0,\tau)},$ and let  $h \in L^1_{loc}([0,\tau))$ such that
$\, k*h \ne 0$ in $L^1_{loc}([0,\tau))$. Then $A$ is a generator of an $(h*k)-$convoluted cosine function $((h*C_k)(t))_{t \in [0,\tau)}.$
\end{lemma}

\vspace{0.2cm}
The next theorem is the main result in this paper. Note that we may extend the support of the solution from $[0, \tau)$ to $[0, n\tau)$ for any $n\in \mathbb{N}$. This technique, based in results given in the previous section, improves the result of Kosti\'c, see \cite[Theorem 3.1]{K2} and \cite[Theorem 2.1.1.14]{ko} in the case $C=I_X$.

\begin{theorem}\label{exten} Let $n \in \mathbb{N},\, 0 < \tau \le \infty, \, k \in L^1([0, (n+1)\tau))$  and $(C_k(t))_{t \in [0,\tau)}$ is  a local $k-$convoluted cosine function generated by $A.$ The family of operators $(C_{k^{*(n+1)}}(t))_{t \in [0,(n+1)\nu]}$ defined for $t \in [0, n\nu]$ by
\begin{equation}{\label{Ck0nu}}
C_{k^{*(n+1)}}(t)x = \int_0^t k(t-r) C_{k^{*n}}(r)x \, dr\,,
\end{equation}
\noindent and for  $t \in [n\nu, (n+1)\nu]$ as
\begin{equation}{\label{CkNu}}
  \begin{array}{rl}
C_{k^{*(n+1)}}(t)x = & \displaystyle  2 C_{k^{*n}}(n\nu) C_{k}(t - n\nu)x  +  \int_0^{n\nu} k(t-r) C_{k^{*n}}(r) x \, dr  \\
& \\
  & + \ \displaystyle  \int_0^{t - n\nu} k^{*n}(t-r) C_{k}(r) x \, dr  -  \int_{2n\nu - t}^{n\nu} k(r + t -2 n\nu ) C_{k^{*n}}(r) x \, dr  \\
& \\
 & - \  \displaystyle \int_0^{t-n\nu} k^{*n}(r-t+2 n \nu) C_{k}(r) x \, dr,\\
\end{array}
\end{equation}

\noindent  is a local $k^{*(n+1)}-$convoluted cosine function with generator $A$ for any $\nu < \tau$ and  $x \in X$.

\end{theorem}

\begin{proof}
Is simple verify that $k^{*(n+1)} \in L^1_{loc}([0,(n+1)\tau))$ and that $\, k^{*(n+1)}\,$ is not identical to zero.
$(C_{k^{*(n+1)}}(t))_{t \in [0,(n+1)\nu]}$ is a strongly continuous operator family which commutes with $A.$
By Proposition \ref{propK}  one gets that $((k*C_{k^{*n}}(t))_{t \in [0,n\nu]}$ is a local $k^{*(n+1)}-$
convoluted cosine function with generator $A$ and consequently (ii) of Definition \ref{cos-con} holds for all $t \in [0, n\nu]$
and $x \in X.$ It remains to be show that this condition is true for every $t \in [n\nu, (n+1)\nu]$ and $x \in X.$
 Write
\begin{equation}{\label{intcompl}}
\int_0^t (t-s) C_{k^{*(n+1)}}(s) x \, ds = \int_0^{n\nu} (t-s) C_{k^{*(n+1)}}(s) x \, ds + \int_{n\nu}^t (t-s) C_{k^{*(n+1)}}(s) x \, ds := J + I,
\end{equation}
\noindent but, $
\begin{array}{l} \displaystyle J =  \displaystyle  \int_0^{n\nu} (n\nu-s) C_{k^{*(n+1)}}(s) x \, ds +
(t- n\nu) \int_0^{n\nu}  C_{k^{*(n+1)}}(s) x \, ds. \\
\end{array}
$
\medskip

 Let  $I = I_1 +  I_2 + I_3 -  I_4  - I_5$ where
 $$
\begin{array}{l}
I_1  =
\displaystyle  2 C_{k^{*n}}(n\nu) \int_{n\nu}^t (t-s) C_{k}(s - n\nu)x \, ds\,, \  \ \ \
 I_2 =  \displaystyle \int_{n\nu}^t (t-s)\int_0^{n\nu} k(s-r) C_{k^{*n}}(r) x \, dr \ ds, \\
 \\
 I_3 =  \displaystyle \int_{n\nu}^t (t-s) \int_0^{s - n\nu} k^{*n}(s-r) C_{k}(r) x \, dr \ ds, \\
 \\
I_4 =  \displaystyle \int_{n\nu}^t (t-s) \int_{2n\nu - s}^{n\nu} k(r + s - 2n\nu ) C_{k^{*n}}(r) x \, dr \ ds, \\
 \\
I_5 =  \displaystyle \int_{n\nu}^t (t-s) \int_0^{s-n\nu} k^{*n}(r-s+ 2n\nu) C_{k}(r) x \, dr \ ds.
\end{array}
$$

\noindent We have $I_1 = \displaystyle 2 C_{k^{*n}}(n\nu) \int_0^{t-n\nu} (t- n\nu -s) C_{k}(s)x \, ds.$ For $I_2$ note that
$$
I_2 = \displaystyle \int_0^{n\nu} C_{k^{*n}}(r) x \, \int_{n\nu}^t (t-s) k(s-r)  \, ds \, dr.
$$
From the simple equality
$$\displaystyle \int_{n\nu}^t (t-s) k(s-r) ds = (I*k)(t - r) - (t - n\nu)(\chi * k)(n\nu - r) -  (I*k)(n\nu - r),$$
\noindent it follows that
$$
\begin{array}{ll}
I_2  = & \displaystyle \int_0^{n\nu} (I*k)(t - r) C_{k^{*n}}(r) x \, dr - \int_0^{n\nu} (t - n\nu)(\chi * k)(n\nu - r) C_{k^{*n}}(r)  x \, dr
\\
&\\
& - \displaystyle \int_0^{n\nu} (I*k)(n\nu - r)  C_{k^{*n}}(r) x \, dr = I_{2_1} - I_{2_2} - I_{2_3}.
\end{array}
$$

First, we compute $I_{2_1}.$
$$
\begin{array}{l}
I_{2_1} = \displaystyle \int_0^{n\nu} C_{k^{*n}}(r) x \, \int_0^{t-r} (t - r - s) k(s) \,  ds \, dr \\
\\
=  \displaystyle \int_0^{t - n\nu} k(s) \int_0^{n\nu} (t-r - s) C_{k^{*n}}(r) x \, dr \, ds \   + \  \displaystyle \int_{t - n\nu}^t k(s) \int_0^{t-s} (t-r-s) C_{k^{*n}}(r)x \, dr \, ds, \\
\end{array}
$$

\noindent with  change of variable $y = t-s$ we have
$$
 \displaystyle \int_{t - n\nu}^t k(s) \int_0^{t-s} (t-r-s) C_{k^{*n}}(r)  x \, dr \, ds = \displaystyle \int_0^{n\nu} k(t - y) \int_0^{y} (y -r) C_{k^{*n}}(r) x \, dr \, dy,
$$
on the other hand,
$$
\begin{array}{l}
\displaystyle \int_0^{t - n\nu} k(s) \int_0^{n\nu} (t-r-s) C_{k^{*n}}(r) x \, dr \, ds \   = \  \displaystyle  \int_0^{n\nu} C_{k^{*n}}(r)x \int_0^{t - n\nu}(t-r-s)  k(s)  \, ds \, dr \\
 \\
= \displaystyle (I*k)(t- n\nu) \int_0^{n\nu} C_{k^{*n}}(r)x \, dr   +  \displaystyle (\chi * k)(t - n\nu) \int_0^{n\nu}(n\nu - r) C_{k^{*n}}(r)x  \, dr. \\
\end{array}
$$

Hence,
$$
\begin{array}{rl}
I_{2_1}  = &   \displaystyle  \int_0^{n\nu} k(t - y) (I*C_{k^{*n}})(y) x \, dy \   +  \displaystyle (I*k)(t- n\nu) (\chi*C_{k^{*n}})(n\nu)   \\
& \\
& + \ \displaystyle (\chi * k)(t - n\nu) (I*C_{k^{*n}})(n\nu)x.
\end{array}
$$

\noindent Now, for $I_{2_2}$  from  (\ref{Ck0nu})  observe that
$$
I_{2_2} =  (t - n\nu) ((\chi * k)*C_{k^{*n}})(n\nu)x  =  (t - n\nu) (\chi * k*C_{k^{*n}})(n\nu)x = (t - n\nu) (\chi*C_{k^{*(n+1)}})(n\nu)x\,.
$$

We obtain that
$$
\begin{array}{rl}
I_2  = &   \displaystyle  \int_0^{n\nu} k(t - y) (I*C_{k^{*n}})(y) x\, dy \  + \  \displaystyle (I*k)(t- n\nu)  (\chi*C_{k^{*n}})(n\nu)x\,  \\
 \\
&  +  \ \displaystyle (\chi * k)(t - n\nu) (I*C_{k^{*n}})(n\nu)x   - (t - n\nu) (\chi*C_{k^{*(n+1)}})(n\nu)x- \displaystyle  (I*k*C_{k^{*n}})(n\nu)x. \\
\end{array}
$$
\noindent We compute $I_3,$ using  Fubini theorem and integrate by parts,
$$
 \begin{array}{rcl}
I_3 & = & \displaystyle \int_0^{t-n\nu} C_{k}(r) x  \int_{r + n\nu}^t (t-s) k^{*n}(s-r)  \, ds \ dr \\
 & = & \displaystyle \int_0^{t-n\nu} C_{k}(r) x  \int_{n\nu}^{t-r} (t-s -r) k^{*n}(s)  \, ds \ dr \\
  & = & \displaystyle \int_0^{t-n\nu} [(I* k^{*n})(t-r) - (I*k^{*n})(n\nu)]  C_{k}(r) x  \, dr
    - (\chi * k^{*n})(n\nu) \  (I*C_{k})(t-n\nu) \\
    & = &  \displaystyle \int_0^{t - n\nu} \int_0^rC_k(s)ds \ (\chi * k^{*n})(t-r)dr - (\chi * k^{*n})(n\nu) \  (I*C_{k})(t-n\nu) \\
    & = &  (\chi * k^{*n})(n\nu) \  (I*C_{k})(t-n\nu) + \displaystyle \int_0^{t - n\nu}(I*C_k)(r) \ k^{*n}(t-r)dr \\
     && - \   (\chi * k^{*n})(n\nu) \  (I*C_{k})(t-n\nu).
\end{array}
$$
Hence $\, I_3   = \displaystyle \int_0^{t - n\nu}(I*C_k)(r) \ k^{*n}(t-r)dr.$

For $I_4,$ using Fubini theorem and change the variable, we have
$$I_4 =  \displaystyle \int_{2n\nu-t}^{n\nu} C_{k^{*n}}(r) x  \int_{2n\nu - r}^{t} (t-s) k(r + s - 2n\nu ) \, ds \ dr =
\displaystyle \int_{2n\nu-t}^{n\nu}  (I*k)(t + r - 2n\nu) C_{k^{*n}}(r) x dr.$$
We integrate twice by parts to obtain that
$$
\begin{array}{rl}
I_4 = &  (I*k)(t - n\nu) \displaystyle \int_0^{n\nu}  C_{k^{*n}}(r) x dr -  \int_{2n\nu - t}^{n\nu} (\chi * k)(t + r - 2n\nu)  \int_0^r C_{k^{*n}}(\tau) x d\tau \ dr\\
& \\
= &  (I*k)(t - n\nu) (\chi*C_{k^{*n}})(n\nu) x - \displaystyle  (\chi * k)(t - n\nu) (I* C_{k^{*n}})(n\nu)  \\
 &\\
 & + \ \displaystyle  \int_{2n\nu - t}^{n\nu} k(t + r - 2n\nu)  \int_0^r (r - \tau) C_{k^{*n}}(\tau) x d\tau \ dr.\\
\end{array}
$$
Applying the same argumentation, for $I_5,$ again by Fubini
$$I_5 =  \displaystyle \int_0^{t - n\nu} C_{k}(r) x \int_{n\nu+u}^t (t-s) k^{*n}(u + 2n\nu - s)  \, ds \ dr.$$
With the change of variable $s = u + 2n\nu - s$ and integration by parts, we have

$$
\begin{array}{rl}
I_5  = & (\chi * k^{*n})(n\nu) \displaystyle \int_0^{t - n\nu}  (t - n\nu - r)  C_{k}(r) x  \, dr -  (I*k^{*n})(n\nu) \displaystyle \int_0^{t - n\nu}   C_{k}(r) x  \, dr\\
   & + \  \displaystyle \int_0^{t - n\nu}   (I*k^{*n})(r + 2n\nu - t) C_{k}(r) x  \, dr  \\
 = & (\chi * k^{*n})(n\nu) \displaystyle \int_0^{t - n\nu}  (t - n\nu - r)  C_{k}(r) x  \, dr -  (I*k^{*n})(n\nu) \displaystyle \int_0^{t - n\nu}   C_{k}(r) x  \, dr\\
  & + \  (I*k^{*n})(n\nu) \displaystyle \int_0^{t - n\nu}   C_{k}(r) x  \, dr -  \displaystyle \int_0^{t - n\nu}  (\chi * k^{*n})(r + 2n\nu - t) \,  \int_0^r C_{k}(s) x ds \ dr\\
  = & (\chi * k^{*n})(n\nu) \displaystyle \int_0^{t - n\nu}  (t - n\nu - r)  C_{k}(r) x  \, dr  - \displaystyle (\chi * k^{*n})(n\nu) \int_0^{t - n\nu}  \, (t - n \nu  - s)C_{k}(s) x ds \\
   & + \  \displaystyle \int_0^{t - n\nu} k^{*n}(r + 2n\nu - t) \,  \int_0^r (r - s)C_{k}(s) x ds \ dr.\\
 \end{array}
$$
\noindent Hence $\, I_5 =  \displaystyle \int_0^{t - n\nu} k^{*n}(r + 2n\nu - t) \, (I*C_{k})(r) \ dr.$ Adding $I_1 +  I_2 + I_3 -  I_4  - I_5,$ it follows that
$$
\begin{array}{rl}
I = & \displaystyle 2\,  C_{k^{*n}}(n\nu) (I*C_{k})(t- n\nu)x  + \displaystyle 2 (\chi * k)(t - n\nu)(I*C_{k^{*n}})(n\nu)x
    \\
& \\
  & -  \ (t - n\nu) (\chi*C_{k^{*(n+1)}})(n\nu) + \displaystyle  \int_0^{n\nu} [k(t - y) - k(n\nu -y)](I*C_{k^{*n}})(y)x dy
   \\
   & \\
     & +  \ \displaystyle \int_0^{t - n\nu}(I*C_k)(r) \ k^{*n}(t-r)dr  -
     \displaystyle  \int_{2n\nu - t}^{n\nu} k(t + r - 2n\nu) (I*C_{k^{*n}})(r)x  \, dr \\
     & \\
    &  - \  \displaystyle \int_0^{t - n\nu} k^{*n}(r + 2n\nu - t) \, (I*C_{k})(r)x  \, dr.

\end{array}
$$

From (\ref{intcompl}), for every $t \in [n\nu, (n+1)\nu]$ and $x \in X,$  we obtain

\vspace{0.2cm}

$$
\begin{array}{l}
\displaystyle \int_0^t (t-s) C_{k^{*(n+1)}}(s) x \, ds=  \displaystyle  \int_0^{n\nu} (n\nu-s) C_{k^{*(n+1)}}(s) x \, ds +
\displaystyle 2 C_{k^{*n}}(n\nu) (I*C_{k})(t- n\nu)x  \\
 \\
 +  \ \displaystyle 2 \, (\chi * k)(t - n\nu)(I*C_{k^{*n}})(n\nu)x +  \displaystyle \int_0^{t - n\nu}(I*C_k)(r) \ k^{*n}(t-r)dr    \\
 \\
 +  \displaystyle  \int_0^{n\nu} [k(t - y) - k(n\nu -y)](I*C_{k^{*n}})(y)x dy
  -   \displaystyle  \int_{2n\nu - t}^{n\nu} k(t + r - 2n\nu) (I*C_{k^{*n}})(r)x  \, dr \\
   \\
    -   \displaystyle \int_0^{t - n\nu} k^{*n}(r + 2n\nu - t) \, (I*C_{k})(r)x  \, dr.
\end{array}
$$

The last equality implies $\displaystyle \int_0^t (t-s) C_{k^{*(n+1)}}(s) x \, ds \in D(A)$ and

$$
\begin{array}{rl}
A\displaystyle \int_0^t (t-s) C_{k^{*(n+1)}}(s) x \, ds= & C_{k^{*(n+1)}}(t)x  - 2 (\chi * k)(t-n\nu) \ (\chi * k^{*n})(n\nu) \\
- & \displaystyle \int_0^{t - n\nu} k^{*n}(t-r)(\chi * k)(r) dr  - \int_0^{n\nu} k(t-r)(\chi * k^{*n})(r) dr \\
+ & \displaystyle \int_{2n\nu - t}^{n\nu} k(r + 2n\nu - t) \, (\chi * k^{*n})(r)x  \, dr \\
 + & \displaystyle \int_0^{t-n\nu}k^{*n}(r + 2n\nu - t) \, (\chi * k)(r)x  \, dr.
\end{array}
$$
From Lemma \ref{lemaclave} we obtain that $\, A\displaystyle \int_0^t (t-s) C_{k^{*(n+1)}}(s) x \, ds=  C_{k^{*(n+1)}}(t)x - (\chi * k^{*(n+1)})(t)x$ for $t\ge 0$ and $x \in X.$
\end{proof}

\begin{remark}\label{c-op} {\rm  Local $k$-convoluted $C$-cosine functions were introduced and studied by Kosti\'c in \cite{K1, K2, ko} with $C$ an injective operator. The main idea is to compose  a local $k$-convoluted cosine function (possibly a family of unbounded operators) with  an injective operator $C$ to get a family of bounded operators, i.e, we have to replace the condition (ii) in Definition \ref{cos-con} by
$$
A \int_0^t(t-s) C_k(s)x ds = C_k(t)x - (\chi * k)(t)Cx, \qquad x\in X, \quad t\in [0, \tau),
$$
and to add a commutative condition, $C_k(t)C=CC_k(t)$,  for $t\in [0, \tau)$, see \cite[Definition 1.1]{K2}. Note that for $C=I_X$, we obtain the local $k$-convoluted cosine function.

A similar result to Theorem \ref{exten} for $k$-convoluted $C$-cosine function holds and the proof is left to the reader. The main point is to express the extension formula given in (\ref{CkNu}) and in this case,
$$
  \begin{array}{rl}
C_{k^{*(n+1)}}(t)x = & \displaystyle  2 C_{k^{*n}}(n\nu) C_{k}(t - n\nu)x  +  \int_0^{n\nu} k(t-r) C_{k^{*n}}(r) Cx \, dr  \\
& \\
 & + \ \displaystyle  \int_0^{t - n\nu} k^{*n}(t-r) C_{k}(r)C x \, dr  -  \int_{2n\nu - t}^{n\nu} k(r + t -2 n\nu ) C_{k^{*n}}(r) Cx \, dr  \\
& \\
  & - \ \displaystyle \int_0^{t-n\nu} k^{*n}(r-t+2 n \nu) C_{k}(r)C x \, dr.\\
\end{array}
$$
Note that we recover  the extension theorem given  \cite[Theorem 3.1]{K2} and \cite[Theorem 2.1.1.14]{ko}  for   $n=1$. Moreover, if we iterate  \cite[Theorem 3.1]{K2}, then we obtain only the extension for $n=2^m$ and $m\in \N$. By uniqueness  of solution of problem (\ref{(ACP_2)}) both expressions coincides.

Note that the extension problem   for $C$-cosine functions may be also found in \cite[Theorem 3.6, Corollary 3.8]{[WG]}.
}
\end{remark}

In the particular case of $k=j_\alpha$ we obtain the next result for $\alpha$-times integrated cosine functions.

\begin{corollary}
Let $n \in \mathbb{N}, 0 < \tau \le \infty$ and $(C_{\alpha}(t))_{t \in [0, \tau)}$ is a local $\alpha-$times integrated cosine function generated by $A$. The family of operators  $(C_{\alpha(n+1)}(t))_{t \in [0, (n+1) \nu]}$
defined for $t \in [0, n\nu]$ by
\begin{equation}{\label{Ck0nu}}
C_{\alpha(n+1)}(t)x = \int_0^t \frac{(t-r)^{\alpha -1}}{\Gamma(\alpha)} C_{\alpha n}(r)x \, dr\,,
\end{equation}
\noindent and for  $t \in [n\nu, (n+1)\nu]$ as
\begin{equation}{\label{CkNu}}
  \begin{array}{rl}
C_{\alpha(n+1)}(t)x = & \displaystyle  2 C_{\alpha n}(n\nu) C_{\alpha}(t - n\nu)x  +  \int_0^{n\nu} \frac{(t - r)^{\alpha - 1 }}{\Gamma(\alpha)} C_{\alpha n}(r) x \, dr  \\
& \\
 & + \ \displaystyle  \int_0^{t - n\nu} \frac{(t-r)^{n\alpha -1}}{\Gamma(n\alpha)} C_{\alpha}(r) x \, dr  -  \int_{2n\nu - t}^{n\nu} \frac{(r + t -2 n\nu )^{\alpha-1}}{\Gamma(\alpha)} C_{\alpha n}(r) x \, dr  \\
& \\
 & - \  \displaystyle \int_0^{t-n\nu} \frac{(r-t+2 n \nu)^{n\alpha -1}}{\Gamma(n\alpha)} C_{\alpha}(r) x \, dr,\\
\end{array}
\end{equation}
\noindent  is a local  $\alpha(n+1)-$times  integrated cosine function with generator $A$ for any $\nu < \tau$ and  $x \in X$.
\end{corollary}

\section{Algebra homomorphisms defined by convoluted cosine functions}
\setcounter{theorem}{0}
\setcounter{equation}{0}

In this section, we show that local $k$-convoluted cosine functions  are kernels to define algebra homomorphism from a certain set of test functions $\mathcal{D}_{k^{\ast\infty}}$ (Theorem \ref{main2}). Note that the extension theorem (Theorem \ref{exten}) is necessary to define the algebra homomorphisms from functions defined on $\R^+$. The set $\mathcal{D}_{k^{\ast\infty}}$ was introduced  in the context of local $k$-convoluted semigroups and $k$-distribution semigroups in $\mathcal{D}_{k^{\ast\infty}}$. However the cosine convolution $\ast_c$ (instead of  $\ast$ convolution in semigroup setting) is the product which we must consider in the cosine setting, see Remarks \ref{notas} (iii).

\begin{theorem} {\label{main2}} Let $k\in L^1_{loc}(\mathbb{R}^+)$ with $0\in \hbox{supp}(k)$,
and $(C_k(t))_{t \in [0,\tau]}$ a non-degenerate local  $k$-convoluted cosine
function generated by $A$.  We define the map ${\mathcal C}_{k}: \mathcal{D}_{k^{\ast\infty}}\to
{\mathcal B}(X)$ by
$$
{\mathcal C}_{k}(f)x:=\int_0^{n\tau} W_{k^{\ast n}} f(t)C_{k^{\ast n}}(t)xdt, \qquad x\in X, \hbox{supp(}f)\subset [0,n\tau], n\in \N,
$$
  where $(C_{k^{\ast n}}(t))_{t \in [0,n\tau]}$ is defined in Theorem \ref{exten}. Then the following properties hold.

\begin{itemize}
\item[(i)] The map ${\mathcal C}_k$ is well defined, linear and bounded.
\item[(ii)] For $\phi, \psi\in  \mathcal{D}_{k^{\ast\infty}}$, we get that
$${\mathcal C}_k(\phi\ast_c \psi)={\mathcal C}_k(\phi){\mathcal C}_k(\psi).
$$
\item[(iii)] ${\mathcal C}_k(f)x\in D(A)$ and $A {\mathcal C}_k(f)x= {\mathcal C}_k(f'')x+f'(0)x$ for any $f\in \mathcal{D}_{k^{\ast\infty}}$ and $x\in X$.
\end{itemize}
\end{theorem}
\begin{proof}

First we  prove (i). Take $f\in \mathcal{D}_{k^{\ast \infty}}$ and $supp(f)\subset [0, n\tau]$ for some $n\in \N$. Let $m\ge n$, $k^{\ast m}= k^{\ast n}\ast k^{\ast (m-n)}$, and   $k^{\ast(m-n)}\circ W_{k^{\ast m}}f =W_{k^{\ast n}}f  $ and $supp(W_{k^{\ast m}}f)\subset [0, n\tau]$
by Proposition \ref{equivalencia} (ii) and (i) respectively. By Lemma \ref{propK} and the Fubini theorem, we get that
\begin{eqnarray*}
&\,&\int_0^{m\tau} W_{k^{\ast m}} f(t)C_{k^{\ast m}}(t)xdt=
\int_0^{n\tau} W_{k^{\ast m}} f(t)(k^{\ast(m-n)}\ast C_{k^{\ast n}})(t)xdt\cr
&\,&=\int_0^{n\tau}\left( k^{\ast(m-n)}\circ W_{  k^{\ast m}} f\right)(t) C_{k^{\ast n}}(t)xdt= \int_0^{n\tau} W_{k^{\ast n}} f(t)C_{k^{\ast n}}(t)xdt
\end{eqnarray*}
for $x\in X$ and we conclude  that ${\mathcal C}_k$ is well defined. It is direct to check that ${\mathcal C}_k$ is  linear and bounded.

Take $f, g \in \mathcal{D}_{k^{\ast \infty}}$. By \cite[Theorem 2.10]{KLM} we have that  $f\ast_c g\in \mathcal{D}_{k^{\ast n}}$ for $n \ge 1$ and then
   $f\ast_c g\in \mathcal{D}_{k^{\ast \infty}}$. Next  we show that ${\mathcal C}_k(f\ast_c g )={\mathcal C}_k(f)
{\mathcal C}_k(g)$. Take $n\in \N$ such that $supp(f), supp(g)\subset [0,n \tau]$ and by Proposition \ref{equivalencia} (i), $supp(W_{k^{\ast 2n}}f), supp(W_{k^{\ast 2n}}g)\subset [0,n \tau]$. Then $supp(f\ast g)\subset [0,2n \tau]$ and $supp(W_{k^{\ast 2n}}(f\ast g))\subset [0,2n \tau]$.  By the definition  of the cosine product $\ast_c$ (\ref{cosine}) and  \cite[Theorem 2.10]{KLM} we have that
$$
{{\mathcal C}}_k(f\ast_c g)x ={1\over 2}\int_0^{2n\tau}\left(W_{k^{\ast 2n}}(f\ast g)+ f\circ
W_{k^{\ast 2n}}g + g\circ W_{k^{\ast 2n}}f\right)(t)C_{k^{\ast 2n}}(t)xdt,
$$
\noindent for $x \in X.$ Again by \cite[Theorem 2.10]{KLM}, we obtain that

\vspace{0.2cm}

$\displaystyle \int_0^{2n\tau}W_{k^{\ast 2n}}(f*g)(t) C_{k^{\ast 2n}}(t)x$
$$
\begin{array}{rl}
 = & \displaystyle \int_0^{2n\tau} C_{k^{\ast 2n}}(t)x \int_0^t W_{k^{\ast 2n}} g(r) \int_{t-r}^t{k^{\ast 2n}}(s+r-t) W_{k^{\ast 2n}} f(s) \, ds dr \,  dt \\
& \\
 & - \ \displaystyle \int_0^{2n\tau} C_{k^{\ast 2n}}(t)x \int_t^{\infty} W_{k^{\ast 2n}} g(r) \int_{t}^{\infty}{k^{\ast 2n}}(s+r-t) W_{k^{\ast 2n}} f(s) \, ds dr \,  dt.\\
\end{array}
$$
 We apply   Fubini theorem to get

\vspace{0.2cm}

$\displaystyle \int_0^{2n\tau} W_{k^{\ast 2n}}(f\ast
g)(t)C_{k^{\ast 2n}}(t)xdt $
$$
\begin{array}{rl}
= & \displaystyle \int_0^{2n\tau}W_{k^{\ast 2n}}g(r)\int_0^{r}W_{k^{\ast 2n}}f(s)
\left(\int_{r}^{s+r}- \int_{0}^{s}{k^{\ast 2n}}(s+r-t)C_{k^{\ast 2n}}(t)xdt\right)dsdr
\\
& + \ \displaystyle \int_0^{2n\tau}W_{k^{\ast 2n}}g(r)\int_r^{2n\tau - r}W_{k^{\ast 2n}}f(s)
\left(\int_{s}^{s+r}-
\int_{0}^{r}{k^{\ast 2n}}(s+r-t)C_{k^{\ast 2n}}(t)xdt\right)dsdr\\
& \\
 & + \  \displaystyle \int_0^{2n\tau}W_{k^{\ast 2n}}g(r)\int_{2n\tau - r}^{2n\tau} W_{k^{\ast 2n}}f(s)
\int_{s}^{2n\tau}
{k^{\ast 2n}}(s+r-t)C_{k^{\ast 2n}}(t)xdt \, ds dr.
\end{array}
$$

For the last integral in the above  equality,  if $r \le n\tau$ then $\displaystyle 2n\tau - r \ge n{\tau},$ and
$$
\begin{array}{l}
\displaystyle \int_0^{n\tau}W_{k^{\ast 2n}}g(r)\int_{2n\tau - r}^{2n\tau} W_{k^{\ast 2n}}f(s)
\int_{s}^{2n\tau}
{k^{\ast 2n}}(s+r-t)C_{k^{\ast 2n}}(t)xdt \, ds dr  = 0 \\
\end{array}
$$

\noindent since $\hbox{supp}(W_{k^{\ast 2n}} f) \subset [0,n\tau].$ Analogously, if $r > n\tau$ then
$$
\begin{array}{l}
\displaystyle \int_{n\tau}^{2n\tau}W_{k^{\ast 2n}}g(r)\int_{2n\tau - r}^{2n\tau} W_{k}f(s)
\int_{s}^{2n\tau}
{k^{\ast 2n}}(s+r-t)C_{k^{\ast 2n}}(t)xdt \, ds dr  = 0 \\
\end{array}
$$
since $\hbox{supp}(W_{k^{\ast 2n}} g) \subset [0,n\tau].$ Hence

\vspace{0.2cm}
$\displaystyle \int_0^{2n\tau} W_{k^{\ast 2n}}(f\ast
g)(t)C_{k^{\ast 2n}}(t)xdt $
$$
\begin{array}{rl}
= & \displaystyle \int_0^{n\tau}W_{k^{\ast 2n}}g(r)\int_0^{r}W_{k^{\ast 2n}}f(s)
\left(\int_{r}^{s+r}- \int_{0}^{s}{k^{\ast 2n}}(s+r-t)C_{k^{\ast 2n}}(t)xdt\right)dsdr
\\
 & + \ \displaystyle \int_0^{n\tau}W_{k^{\ast 2n}}g(r)\int_r^{n\tau }W_{k^{\ast 2n}}f(s)
\left(\int_{s}^{s+r}-
\int_{0}^{r}{k^{\ast 2n}}(s+r-t)C_{k^{\ast 2n}}(t)xdt\right)dsdr.\\
\end{array}
$$

\vspace{0.2cm}

On the other hand, by the Fubini theorem we get that

$$
\int_0^{2n\tau} (f\circ
W_{k^{\ast 2n}}g)(t)C_{k^{\ast 2n}}(t)xdt = \int_0^{2n\tau}W_{k^{\ast 2n}}g(r)\int_0^r
f(r-t)C_{k^{\ast 2n}}(t)xdtdr
$$

\noindent  since $f(r-t)=({k^{\ast 2n}}\circ W_{k^{\ast 2n}}f)(r-t)$, we have  that

$\displaystyle \int_0^{2n\tau} (f\circ W_{k^{\ast 2n}}g)(t)C_{k^{\ast 2n}}(t)xdt$
$$
\begin{array}{rl}
 = & \displaystyle \int_0^{2n\tau}W_{k^{\ast 2n}}g(r)\int_0^{r}C_{k^{\ast 2n}}(t)x \int_{r-t}^{2n\tau}{k^{\ast 2n}}(s-r+t)W_{k^{\ast 2n}}f(s) ds dt dr \\
& \\
= & \displaystyle \int_0^{n\tau}W_{k^{\ast 2n}}g(r)\int_0^{r}W_{k^{\ast 2n}}f(s)\int_{r-s}^{r}{k^{\ast 2n}}(s-r+t)C_{k^{\ast 2n}}(t)xdtdsdr
\\
& \\
 & + \ \displaystyle \int_0^{n\tau}W_{k^{\ast 2n}}g(r)\int_r^{n\tau}W_{k^{\ast 2n}}f(s)
\int_{0}^{r}{k^{\ast 2n}}(t+s-r)C_{k^{\ast 2n}}(t)xdtdsdr.\\
\end{array}
$$
In similar way, we also get that

$\displaystyle \int_0^{2n\tau} (g\circ W_{k^{\ast 2n}}f)(t)C_{k^{\ast 2n}}(t)xdt$
$$
\begin{array}{rl}
 =
& \displaystyle \int_0^{2n\tau}W_{k^{\ast 2n}}f(r)\int_0^{r}W_{k^{\ast 2n}}g(s)\int_{r-s}^{r}{k^{\ast 2n}}(s-r+t)C_{k^{\ast 2n}}(t)xdtdsdr
\\
& \\
 & + \  \displaystyle \int_0^{2n\tau}W_{k^{\ast 2n}}f(r)\int_r^{2n\tau}W_{k^{\ast 2n}}g(s)
\int_{0}^{r}{k^{\ast 2n}}(s-r + t)C_{k^{\ast 2n}}(t)xdtdsdr.\\
\end{array}
$$
We apply Fubini theorem,  change the variable and use that $\hbox{supp}(W_{k^{\ast 2n}} f), \hbox{supp}(W_{k^{\ast 2n}} g) \subset [0,n\tau]$ to obtain that

$\displaystyle \int_0^{2n\tau} (g\circ W_{k^{\ast 2n}}f)(t)C_{k^{\ast 2n}}(t)xdt$
$$
\begin{array}{rl}
= & \displaystyle \int_0^{n\tau}W_{k^{\ast 2n}}g(r)\int_{r}^{n\tau} W_{k^{\ast 2n}}f(s)\int_{s-r}^{s}{k^{\ast 2n}}(r-s+t)C_{k^{\ast 2n}}(t)xdtdsdr
\\
& \\
 & + \  \displaystyle \int_0^{n\tau}W_{k^{\ast 2n}}g(r)\int_0^{r}W_{k^{\ast 2n}}f(s)
\int_{0}^{s}{k^{\ast 2n}}(r - s + t)C_{k^{\ast 2n}}(t)xdtdsdr.\\
& \\
\end{array}
$$
We join these six summands and obtain
$$
\begin{array}{l}
2 \, {{\mathcal C}}_k(f\ast_c g)x =  \displaystyle \int_0^{n\tau}W_{k^{\ast 2n}}g(r)\int_0^{r}W_{k^{\ast 2n}}f(s)
\left(\int_{r}^{s+r}- \int_{0}^{s}{k^{\ast 2n}}(s+r-t)C_{k^{\ast 2n}}(t)xdt  \right. \\
\\
\ \ \ + \displaystyle \left.  \int_{r-s}^{r}{k^{\ast 2n}}(s-r+t)C_{k^{\ast 2n}}(t)xdt+
  \int_{0}^{s}{k^{\ast 2n}}(r - s + t)C_{k^{\ast 2n}}(t)xdt \right)dsdr \\
\\
\ \ \ + \displaystyle \int_0^{n\tau}W_{k^{\ast 2n}}g(r)\int_r^{n\tau}W_{k^{\ast 2n}}f(s)
\left(\int_{s}^{s+r}-
\int_{0}^{r}{k^{\ast 2n}}(s+r-t)C_{k^{\ast 2n}}(t)xdt\right. \\
\\
\ \ \ + \displaystyle\left.\int_{0}^{r}{k^{\ast 2n}}(t+s-r)C_{k^{\ast 2n}}(t)xdt+\int_{s-r}^{s}{k^{\ast 2n}}(r-s+t)C_{k^{\ast 2n}}(t)xdt\right)dsdr \\
\\
=\displaystyle 2\int_0^{n\tau}W_{k^{\ast 2n}}g(r)\int_0^{n\tau}W_{k^{\ast 2n}}f(s)C_{k^{\ast 2n}}(r)C_{k^{\ast 2n}}(s)xdsdr = 2{\mathcal C}_{k}(f){\mathcal C}_{k}(g)x.
\end{array}
$$
\noindent we apply the composition property of convoluted cosine functions, \cite[Theorem 2.4]{K2}.

To finish the proof consider $f\in \mathcal{D}_{k^{\ast \infty}}$, $supp(f)\subset [0,n\tau]$ and $x\in X$. By  \cite[Lemma 2.8 (i)]{KLM} we have $W_{k^{\ast n}}(f'')= (W_{k^{\ast n}})'' $   and  according to  Definition \ref{cos-con} (ii) we have
\begin{eqnarray*}
A {\mathcal C}_k(f)x&=&A\int_0^{n\tau}(W_{k^{\ast n}}f)''(t)\int_0^{t}(t-s)C_{k^{\ast n}}(s)xdsdt\\
&=&
\int_0^{n\tau}W_{k^{\ast n}}f''(t)\left(C_{k^{\ast n}}(t)x-\int_0^t{k^{\ast n}}(s)dsx\right)dt\\
&=&{\mathcal C}_{k^{\ast n}}(f'')x+
\int_0^{n\tau}W_{k^{\ast n}}f'(t){k^{\ast n}}(t)dtx={\mathcal C}_{k^{\ast n}}(f'')x+f'(0)x,
\end{eqnarray*}
for $x\in X$ and we conclude the proof.
\end{proof}

\begin{remark}\label{notas}{\rm (i) In the conditions of Theorem \ref{main2}, now take $l\in L^1_{loc}(\mathbb{R}^+)$ with $0\in \hbox{supp}(l)$. By  \cite[Lemma 4.4]{K1}, the family  $((l\ast C_{k})(t))_{t\in[0,\tau]}$ is a $ k\ast l$-convoluted cosine function. Then one may prove that $${\mathcal C}_{k\ast l}(f)=
{\mathcal C}_k(f), \qquad  f \in {\mathcal D}_{(k\ast l)^{\ast\infty}}.$$

\medskip

(ii) When the operator $A$ generates a global convoluted cosine function $(C(t))_{t\ge 0}$, the homomorphism ${\mathcal C}_k$ is defined from ${\mathcal D}_k$ to ${\mathcal B}(X)$ by
$$
{\mathcal C}_k(f)x=\int_0^\infty W_kf(t)C_k(t)xdt, \qquad x\in X, \qquad f\in {\mathcal D}_k,
$$
see \cite[Theorem 6.5]{KLM}. If the family  $(C(t))_{t\ge 0}$ is exponentially bounded, then the homomorphism  ${\mathcal C}_k$  is extended to a  bounded Banach algebra homomorphism, see \cite[Theorem 6.7]{KLM}.

\medskip

(iii) Note that in fact the Theorem \ref{main2} motivates the following class of distribution cosine functions. Let $k\in L^{1}_{loc}(\mathbb{R}^{+})$ such that $0\in\textrm{supp}(k)$.  A  linear and continuous map $ {\mathcal C}_k: \mathcal{D}_{k^{\ast\infty}} \mapsto \mathcal{B}(X)$ is said a {\it $k$-distribution cosine function}, in short $k$-(DCF), if  satisfies the following conditions.

\begin{itemize}
\item[(k1)] ${\mathcal C}_k(\phi\ast_c \psi)={\mathcal C}_k(\phi){\mathcal C}_k(\psi)$ for $\phi, \psi\in  \mathcal{D}_{k^{\ast\infty}}.$
\item[(k2)] $\cap \{ \ker({\mathcal C}_k(\theta))\,\, \vert \,\, \theta \in  \mathcal{D}_{k^{\ast\infty}}\}=\{0\}$.
\end{itemize}

\medskip

For a given $k$-(DCF) ${\mathcal C}_k$,  the generator   $(A, D(A))$ of ${\mathcal C}_k$ is defined by
\begin{eqnarray*}
D(A):&=& \{ x\in X \vert \hbox{ exists } y \in X \hbox{ such that } {\mathcal C}_k(\theta)y= {\mathcal C}_k(\theta'')x+\theta'(0)x \hbox{ for any } \theta\in  \mathcal{D}_{k^{\ast\infty}}\};\\
A(x)&=&y, \qquad x\in D(A).
\end{eqnarray*}
The operator  $(A, D(A))$ is well-defined, closed, linear  and
${\mathcal C}_k(\mathcal{D}_{k^{\ast\infty}})X\subset D(A)$ with $A{\mathcal C}_k(\phi)x={\mathcal C}_k(\phi'')x+\phi'(0)x$ for any $\phi \in \mathcal{D}_{k^{\ast\infty}} $ and $x\in X$ (we use Proposition \ref{propi}(v)).

\noindent By Theorem \ref{main2}, a local convoluted cosine function defines a $k$-distribution cosine function.

A very important case is when $\mathcal{D}_k= \mathcal{D}_+$. In this condition we recover the definition of almost-distribution cosine function given in \cite[Definition 5]{M0}. In fact the equivalence between almost-distribution cosine functions, distribution cosine functions (introduced by Kosti\'c in \cite{KTW}) and $n$-times integrated cosine functions (and other spectral conditions) were proven in \cite[Theorem 3.5]{KMTW}.

\medskip

(iv) Some connections between  $k$-convoluted cosine function and  ultradistributions and hyperfunction sines have been treated in \cite[Section 5]{[KP2]} and \cite[Section 3.6.3]{ko}. In particular if $A$ is the generator of an exponentially bounded $k$-cosine function for certain $k$, then there is a ultradistribution fundamental solution of $\ast$-class for $\pm iA$ (\cite[Theorem 20]{[KP2]}). Note that the relationship between this result and Theorem \ref{main2} may be a matter of further investigations, see  \cite[Remark 14]{[KP2]}.

}
\end{remark}

\section{Examples and final comments }
\setcounter{theorem}{0}
\setcounter{equation}{0}

In this section we consider different examples of  convoluted cosine functions which have appeared in the literature.  Our results are applied in these examples to illustrate its importance.

\subsection{The Laplacian on $L^p(\R^N)$} \cite{EK,[Hi], K} It is known that the Laplacian $\Delta$ is the generator of an $\alpha$-times integrated semigroups in $L^p(\R^N)$ with $1<p<\infty$ and $N\ge 1$, $(C_\alpha(t))_{t\ge 0}\subset  {\mathcal B}(L^p(\R^N))$, if and only if $\alpha \ge (N-1)\left\vert {1\over p}-{1\over 2}\right\vert$, and
$$
\Vert C_\alpha (t)\Vert \le Ct^\alpha, \qquad t\ge 0,
$$
see \cite[Theorem 4.4]{[Hi]}, \cite[Proposition 3.2]{EK} and \cite[Section 5]{K}. In this case the map  ${\mathcal C}_\alpha: {\mathcal D}_+\to {\mathcal B}(L^p(\R^N))$ (Remark \ref{notas} (ii)) extends to a Banach algebra homomorphism ${\mathcal C}_\alpha: {\mathcal T}_\alpha(t^\alpha) \to {\mathcal B}(L^p(\R^N))$ where ${\mathcal T}_{\alpha}(t^\alpha)$ is the completion of
${\mathcal D}_+$ in the norm
$$
\Vert f\Vert_{\alpha }:=\int_0^\infty\vert W^\alpha f(t)\vert t^{\alpha}dt, \qquad f\in {\mathcal D}_+,
$$
where $W^\alpha$ is the Weyl derivation of order $\alpha$, see \cite[Theorem 4]{M0}. Other examples of global $\alpha$-times integrated cosine functions (generated by translation invariant operators) may be found in \cite[Theorem 5.2, Theorem 5.4]{EK}, \cite[Theorem 3.1]{K} and \cite[Theorem 4.2]{[Hi]}.

\subsection{The Laplacian on $L^2[0,\pi]$ with Dirichlet  boundary conditions} \cite[Example 1]{[KP2]}, \cite[Example 2.1.8.1]{ko}  The operator $-\Delta$ on $L^2[0,\pi]$ with Dirichlet  boundary conditions generates a exponentially bounded $\kappa$-convoluted cosine function $(C_{\kappa}(t))_{t\ge 0}$ where $\kappa$ is given in Example \ref{function} (v). By Remark \ref{notas} (ii), there exists an algebra homomorphism ${\mathcal C}_\kappa: {\mathcal D}_\kappa \to {\mathcal B}(L^2[0,\pi])$; in fact  it extends to a Banach algebra homomorphism ${\mathcal C}_\kappa: {\mathcal T}_\kappa(e_\beta) \to {\mathcal B}(L^2[0,\pi])$ where ${\mathcal T}_\kappa(e_\beta)$ is the completion of
${\mathcal D}_\kappa$ in the norm
$$
\Vert f\Vert_{\kappa, e_{-\beta}}:=\int_0^\infty\vert W_\kappa f(t)\vert e^{\beta t}dt, \qquad f\in {\mathcal D}_\kappa,
$$
for some $\beta >0$, see  \cite[Corollary 6.6]{KLM}. Moreover, the fractional power $\Delta^{2^n}$ generates an exponentially bounded $\kappa_{n+1}$-convoluted cosine function for suitable kernel $\kappa_{n+1}$ (\cite{ko,[KP2]}). Similar results for  $n\in \N$ are obtained following previous discursion.

\subsection{Multiplication operator in $L^p(\R)$} \cite[Section 6, Example 3]{[KP2]} Take $X=L^p(\R)$  with $1\le p\le \infty$ and we consider the multiplication operator $A$ with the maximal domain in $L^p(\R)$,
$$
Af(x)=(x+ix^2)^2f(x), \qquad x\in \R, f\in L^p(\R).
$$
The operator $A$ is not the generator of any (local) integrated cosine function in $L^p(\R)$ for $1\le p\le \infty$. However $A$ is the generator of a local $K_\delta$-convoluted cosine function $(C_{K_\delta}(t))_{t\in[0, \tau)}$ given by
$$
C_{K_\delta}(t)f(x):={1\over 2\pi i}\int_\Gamma{\lambda e^{\lambda t-\lambda^\delta}\over \lambda^2-(x+ix^2)^2}d\lambda f(x), \qquad f\in L^p(\R),\, x\in \R,\, t\in [0,\tau),
$$
where $K_\delta$ is given in Example \ref{function} (iii) for suitable $\delta \in (0,1)$, a complex path $\Gamma$ and for any $\tau\in (0,\infty)$. Then  $A$ is the generator of a global $K_\delta$-convoluted cosine function $(C_{K_\delta}(t))_{t\in[0, \infty)}$ and  we apply the Remark \ref{notas} (ii) to define an algebra homomorphism ${\mathcal C}_{K_\delta}: {\mathcal D}_{K_\delta}\to {\mathcal B}(X)$. Note that ${\mathcal D}_{K_\delta}\varsubsetneq{\mathcal D}_+$ (in other case $A$ generates a local integrated cosine function, see Remark \ref{notas} (iii)).

\subsection{Multiplication operator in $L^p(\R^+)$}  \cite[Example 6.2]{KTW} Take $X=L^p(\R^+)$  with $1\le p\le \infty$ and we consider the multiplication operator $A$ with the maximal domain in $L^p(\R)$,
$$
Af(x)=(x+ie^x)^2f(x), \qquad x>0, f\in L^p(\R).
$$
The operator $A$ is the generator of (local) $1$-integrated cosine function $(C_1(t))_{t\in[0,1]}$ in $L^p(\R^+)$ for $1\le p\le \infty$, given by
$$
(C_1(t)f)(x)={\sinh((x+ie^x)t)\over x+ie^x}f(x), \qquad t\in [0,1], \,\, x>0, \,\, f\in L^p(\R^+),
$$
for $0\le t<1 $ and $\sup_{t\in [0,1]}\Vert C_1(t)\Vert \le 1$ for $t\in [0,1]$. By Theorem \ref{exten}, the operator $A$ generates a $n$-times integrated cosine function $(C_n(t))_{t\in[0,n]}$ in $L^p(\R^+)$ for $1\le p\le \infty$ and a distribution cosine function ${\mathcal C}_\alpha:{\mathcal D}_+\to {\mathcal B}(X)$.

\subsection{Multiplication operator in $\ell^2$}  \cite[Example 1]{M1}, \cite[Example 1.2.6]{[MF]}
  Let $\ell^2$ be the Hilbert space of all
 sequences $x=(x_m)_{m\in \mathbb{N}}$ such that
 $
 \sum_{m=1}^\infty\vert x_m\vert^2<\infty,
 $
with the usual norm  $\Vert x\Vert:= \left(\sum_{m=1}^\infty\vert
x_m\vert^2\right)^{1\over 2}$. Take $T>0$ and define
$$
a_m={m\over T}+ i\left(\left({e^m\over m}\right)^2-\left({m\over
T}\right)^2\right)^{1\over 2}, \quad m\in \N,
$$
where $i^2=-1$. For any $\alpha >0$ let  $(C_\alpha (t))_{t>0}$ be
defined by
$$
C_\alpha (t)x=\left({1\over \Gamma(\alpha )} \int_0^t(t-s)^{\alpha
-1}\cosh({a_m s})x_m ds\right)_{m=1},
$$
for $x\in  \ell^2.$ Then
$(C_\alpha(t))_{t\in[0, \alpha T)}$ is a local $\alpha $-times
integrated cosine function on $\ell^2$ such that
 $(C_\alpha (t))_{t\in[0,\alpha T)}$ cannot be extended to $t\ge \alpha
 T$, in fact $$C_\alpha(t)={U_\alpha (t)+U_\alpha (-t)\over 2}, \quad t\in[0, \alpha T),$$
 where $(U_\alpha(t))_{t\in [0, \alpha T)}$ are local $\alpha$-integrated semigroups, see \cite[Example 1]{M1}.
 By Theorem \ref{exten} $(C_{n\alpha} (t))$ may be defined to $t<n\alpha T$.


\end{document}